\documentclass{amsart}
\usepackage{amssymb, amsmath, amsthm, url}
\usepackage[dvips]{graphicx}
\usepackage{tikz}\usetikzlibrary{matrix,arrows}

\newtheorem{thm}{Theorem}[section]
\newtheorem{mainthm}{Theorem}

\newtheorem{lem}[thm]{Lemma}
\newtheorem{prop}[thm]{Proposition}
\newtheorem{cor}[thm]{Corollary}

\theoremstyle{definition}
\newtheorem{defn}[thm]{Definition}
\newtheorem{rem}[thm]{Remark}
\newtheorem{exam}[thm]{Example}
\newtheorem{ques}[thm]{Question}

\newcommand{\Star}[1]{\mathrm{St}(#1)}
\newcommand{\Link}[1]{\mathrm{Lk}(#1)}
\newcommand{\gen}[2]{\pi^{#1}_{#2}}
\newcommand{\genprod}[3]{\delta^{#1}_{#2, #3}}
\newcommand{\inner}[1]{\iota^{#1}}
\newcommand{\partialconjugations}{\mathcal{P}}
\newcommand{\rank}{\mathrm{rk}}

\begin{document}
\pagestyle{plain}

\title{The Bieri-Neumann-Strebel Invariant of the Pure Symmetric Automorphisms of a Right-Angled {A}rtin Group}

\author[Koban]{Nic~Koban}
\address{Dept.~of Math., University of Maine Farmington, Farmington ME 04938}
\email{nicholas.koban@maine.edu}

\author[Piggott]{Adam~Piggott}
\address{Dept.~of Math., Bucknell University, Lewisburg PA 17837}
\email{adam.piggott@bucknell.edu}

\date{\today}

\begin{abstract}
We compute the BNS-invariant for the pure symmetric automorphism groups of right-angled Artin groups.  We use this calculation to show that the pure symmetric automorphism group of a right-angled Artin group is itself not a right-angled Artin group provided that its defining graph contains a separating intersection of links.
\end{abstract}

\maketitle

\section{Introduction}
In 1987, the Bieri-Neumann-Strebel (BNS) geometric invariant $\Sigma^1(G)$ was introduced for a discrete group $G$.  The invariant is an open subset of the character sphere $S(G)$ which carries considerable algebraic and geometric information. It determines whether or not a normal subgroup with abelian quotient is finitely generated; in particular, the commutator subgroup of $G$ is finitely generated if and only if $\Sigma^1(G) = S(G)$.  If $M$ is a smooth compact manifold and $G = \pi_1(M)$, then $\Sigma^1(G)$ contains information on the existence of circle fibrations of $M$.  Additionally, if $M$ is a $3$-manifold, then $\Sigma^1(G)$ can be described in terms of the Thurston norm.   Other aspects of the rich theory of BNS-invariant can be found in \cite{bieri}.

Although $\Sigma^1(G)$ has proven quite difficult to compute in general, it has been computed in the case that $G$ is a right-angled Artin group  \cite{meier}, and in the case that $G$ is the pure symmetric automorphism group of a free group \cite{orlandi}.  In the present article we generalize the result of \cite{orlandi} by computing $\Sigma^1(G)$ when $G$ is the pure symmetric automorphism group of a right-angled Artin group. The outcome of the computation is recorded in Theorem \ref{Theorem:MainA}, to be found in \S~\ref{main-a} below.

We also provide an application of our computation.  It was shown in \cite{charney} that if $A$ is the right-angled Artin group determined by a graph $\Gamma$ that has no separating intersection of links (no SILS), then the corresponding group of pure symmetric automorphisms $P\Sigma(A)$ is itself a right-angled Artin group.  We prove the converse by observing that when $\Gamma$ has a SIL, the BNS-invariant of $P\Sigma(A)$ does not have a certain distinctive property that the BNS-invariant of a right-angled Artin group must satisfy.  Thus we prove:

\setcounter{mainthm}{1}
\begin{mainthm}\label{Theorem:MainB}
The group $P\Sigma(A)$ is isomorphic to a right-angled Artin group if and only if the defining graph $\Gamma$ contains no SILs.
\end{mainthm}

Theorem~\ref{Theorem:MainB} is indicative of a dichotomy within the family of groups $\{P\Sigma(A)\}$ determined by whether or not $\Gamma$ has a SIL.  Certain algebraic manifestations of this dichotomy were proved in \cite{gutierrez}.  It would be interesting to understand more geometric manifestations.  Since right-angled Artin groups are CAT(0) groups, we are lead to ask the following question:

\begin{ques}\label{Question:SILsImplyCAT(0)}
If the defining graph $\Gamma$ contains a SIL, is $P\Sigma(A)$ a CAT(0) group?
\end{ques}

This paper is organized as follows: in \S~\ref{BNS} and \S~\ref{psa}, we define the BNS-invariant $\Sigma^1(G)$ and the pure symmetric automorphism group, respectively, and record some useful facts which inform the arguments to follow.  We prove Theorem~\ref{Theorem:MainA} in \S~\ref{main-a}.  This proof involves two cases with the first handled in \S~\ref{typeI} and the second in \S~\ref{typeII}.  In \S~\ref{sils}, we prove Theorem~\ref{Theorem:MainB}.


 \section{The BNS-invariant}\label{BNS}

Let $G$ be a finitely generated group. A {\it character} $\chi$ of $G$ is a homomorphism from $G$ to the additive reals.  The set of all characters of $G$, denoted Hom$(G,\mathbb{R})$, is an $n$-dimensional real vector space where $n$ is the $\mathbb{Z}$-rank of the abelianization of $G$.  Two non-zero characters $\chi_1$ and $\chi_2$ are equivalent if there is a real number $r > 0$ such that $\chi_1 = r\chi_2$.  The set of equivalence classes $S(G) = \{ [\chi]~|~\chi \in {\rm Hom}(G,\mathbb{R}) - \{0\} \}$ is called the {\it character sphere of $G$}, and this is homeomorphic to an $(n-1)$-dimensional sphere. The BNS invariant $\Sigma^1(G)$, a subset of $S(G)$, may be described in terms of either the geometry of Cayley graphs (see \cite{bieri}), or $G$-actions on $\mathbb{R}$-trees (see \cite{Brown}).  For our purposes the latter is more convenient, and we now describe $\Sigma^1(G)$ from that point of view.

Suppose $G$ acts by isometries on an $\mathbb{R}$-tree,
$T$, and let $\ell:G \to \mathbb{R}^+$ be the corresponding length function.  For each $g \in G$,
let $C_g$ be the characteristic subtree of $g$.  If $\ell(g) = 0$, then $g$ is elliptic, and $C_g$ is its fixed point set;
if $\ell(g) \neq 0$, then $g$ is hyperbolic, and $C_g$ is the axis of $g$. The action is {\it non-trivial} if at least one element of $G$ is hyperbolic, and \emph{abelian} if every element of $[G,G]$ is elliptic.   A non-trivial abelian action on an $\mathbb{R}$-tree must fix either one or two ends of the tree, and is considered {\it exceptional} if it fixes only one end.  To each non-trivial abelian action, and each fixed end $e$, we associate the character $\chi$ such that $|\chi(g)| = \ell(g)$, and $\chi(g)$ is positive if and only if $g$ is a hyperbolic isometry which translates its axis away from the fixed end $e$.  
We say $g$ is \emph{$\chi$-elliptic} if $\chi(g)=0$, and \emph{$\chi$-hyperbolic} otherwise.

We are now able to give Brown's formulation of $\Sigma^1(G)$: An equivalence class $[\chi] \in S(G)$ is contained in $\Sigma^1(G)$ unless there exists an $\mathbb{R}$-tree $T$ equipped with an exceptional non-trivial abelian $G$-action associated to $\chi$.

To demonstrate that $[\chi] \in \Sigma^1(G)$, it suffices to show that in any $\mathbb{R}$-tree $T$ equipped
with a non-trivial abelian $G$-action associated to $\chi$, there exists a line $X$ such that $X \subseteq C_g$ for all $g \in G$.
For this purpose, the following facts about characteristic subtrees are invaluable (see  \cite{orlandi}:
\begin{description}
\item [Fact A] If $[g, h]=1$ and $h$ is hyperbolic, then $C_h \subseteq C_g$.
\item [Fact B] If $[g, h]=1$, then $C_g \cap C_h \subseteq C_{gh}$.
\end{description}
Essentially, we work with a fixed finite generating set of $G$, we consider an arbitrary non-trivial abelian $G$-action on an arbitrary $\mathbb{R}$-tree $T$, we let $X \subseteq T$ denote the axis of one $\chi$-hyperbolic generator $s$, and we use Facts A and B to demonstrate that $X \subseteq C_t$ for every other generator $t$.  For this approach to be successful we typically need a sufficient number of commuting relations in $G$.

To demonstrate that $[\chi] \in \Sigma^1(G)^c$, it is often convenient to make use of the following well-known facts.

\begin{lem}\label{Lemma:Complement}
Let $\chi \in {\rm Hom}(G,\mathbb{R})- \{0\}$.  Suppose there is an epimorphism $\phi:G \to H$ and
a character $\psi \in {\rm Hom}(H,\mathbb{R})$ such that $\chi = \psi \circ \phi$.
If $[\psi] \in \Sigma^1(H)^c$, then $[\chi] \in \Sigma^1(G)^c$.
\end{lem}

\begin{cor}\label{Corollary:FreeProduct}
If $A$ and $B$ are non-trivial finitely-generated groups, and $\chi \in {\rm Hom}(G,\mathbb{R})- \{0\}$ factors through an epimorphism $G \to A \ast B$, then $[\chi] \in \Sigma^1(G)^c$.
\end{cor}

\begin{proof}
This follows from Lemma \ref{Lemma:Complement}, and the fact that $\Sigma^1(A \ast B) = \emptyset$.
\end{proof}

 \section{Right-angled Artin groups and their pure symmetric automorphisms}\label{psa}

Throughout we fix a simplicial graph $\Gamma$, with vertex set $V$ and edge set $E$.  For each vertex $a \in V$, the \emph{link of $a$} is the set
$\Link{a} = \{b \in V \; | \; \{a, b\} \in E\}$, and the \emph{star of $a$} is the set $\Star{a} = \Link{a} \cup \{a\}$.  For a set of vertices $W \subseteq V$, we write $\Gamma \setminus W$ for the full subgraph spanned by the
vertices in $V \setminus W$.

Let $A = A(\Gamma)$ denote the right-angled Artin group determined by $\Gamma$.
We shall not distinguish between the vertices of $\Gamma$ and the generators of $A$, thus
$A$ is the group presented by
$$\langle V \; | \; ab = ba \text{ for all } a, b \in V \text{ such that } \{a, b\} \in E \rangle.$$

For each vertex $a \in V\setminus Z$, and each connected component $K$ of $\Gamma \setminus \Star{a}$, the map
$$v \mapsto \left\{
                            \begin{array}{ll}
                              a^{-1} v a & \text{if } v \in K,\\
                              v & \text{if } v \in V\setminus K, \\
                            \end{array}
                          \right.
$$ extends to an automorphism $\gen{a}{K}\!:\! A \to A$. We say $\gen{a}{K}$ is the \emph{partial conjugation (of $A$) with acting letter
$a$ and domain $K$}.  We write $\partialconjugations$ for the set comprising the partial conjugations.

The \emph{pure symmetric automorphism group}, $P\Sigma(A)$, comprises those automorphisms $\alpha\!:\! A \to A$ which map each vertex to a conjugate of itself. Laurence proved that $P\Sigma(A)$ is generated by $\partialconjugations$ \cite{Laurence}.

We let $Z = \{a \in V \; | \; \Star{a} = V\}$,
and we may assume $Z \neq \emptyset$ for the following reason: it follows immediately from Laurence's result, together with the observation that enriching $\Gamma$ with a new vertex $w$ adjacent to all other vertices does not introduce new partial conjugations, and does not change the domain of any existing partial conjugation.   Let $d\!:\! V \times V \to \{0, 1, 2\}$ denote the
combinatorial metric on $V$.

We now record three results, paraphrased from existing literature, which make working with $\partialconjugations$ tractable.
A proof of the first is included because it is so brief; the second follows immediately from the first.

\begin{lem}\label{Lemma:TwoGens}\cite[Lemma 4.3]{gutierrez}
If $\gen{a}{K}, \gen{b}{L} \in \partialconjugations$ and $d(a, b) = 2$ and $b \not \in K$, then either $K \cap L = \emptyset$ or $K \subseteq L$.
\end{lem}

\begin{proof}
Assume $\gen{a}{K}, \gen{b}{L} \in \partialconjugations$ and $d(a, b) = 2$ and $b \not \in K$.  For the sake of contradiction, suppose $\emptyset \neq K \cap L \neq K$.
Let $u \in K \cap L$ and $v \in K \setminus L$.  Since $K$ is connected, there exists a path $\alpha$ in $K$ from $u$ to $v$.  Since $u \in L$ and $v \not \in L$, $\alpha$ passes
through a vertex $w \in \Star{b}$.  Since $d(b, w) \leq 1$ and $w \in K$ and $b \in \Gamma \setminus \Star{a}$, $b \in K$---a contradiction.
\end{proof}

\begin{lem}\label{Lemma:PCCases}\cite[Corollary 4.4 and Lemma 4.7]{gutierrez}
For each pair of partial conjugations $(\gen{a}{K}, \gen{b}{L}) \in \partialconjugations \times \partialconjugations$, exactly one of the following six cases holds:
\begin{enumerate}
\item [(1)] $d(a, b) \leq 1$;
\item [(2)] $d(a, b) = 2$, $a \in L$, and $b \in K$;
\item [(3)] $d(a, b) = 2$, $K \cap L = \emptyset$, and either $a \in L$ or $b \in K$;
\item [(4)] $d(a, b) = 2$, and either $\{a\} \cup K \subset L$ or $\{b\} \cup L \subset K$;
\item [(5)] $d(a, b) = 2$, and $\bigl(\{a\} \cup K\bigr) \cap (\{b\} \cup L) = \emptyset$;
\item [(6)] $d(a, b) = 2$, and $K = L$.
\end{enumerate}
The relation $[\gen{a}{K}, \gen{b}{L}] = 1$ holds only in the cases (1), (4) and (5).
\end{lem}

\begin{thm}\label{Thm:Presentation}\cite[Chapter 3]{toinet}
Every relation between partial conjugations is a consequence of the following relations:
\begin{enumerate}
\item $[\gen{a}{K}, \gen{b}{L}] = 1$ if $(\gen{a}{K}, \gen{b}{L})$ falls into one of the cases (1), (4), (5) of Lemma \ref{Lemma:PCCases};
\item \label{DeltaRelation} $[\gen{a}{K} \gen{a}{L}, \gen{b}{L}] = 1$ if $K \neq L$ and $b \in K$.\end{enumerate}
\end{thm}

It is convenient to introduce notation for certain products of partial conjugations with the same acting letter.  We write $\genprod{a}{K}{L}$ for the product $\gen{a}{K}\gen{a}{L}$, provided $K \neq L$.  We write $\inner{a}$ for the inner automorphism $w \mapsto a^{-1} w a$ for all $w \in A$, and we note $\inner{a}$ is simply the product of all partial conjugations with acting letter $a$.

Next we record some useful facts about the behavior of partial conjugations.

\begin{lem}\label{Lemma:NewPC}
If $\gen{a}{K}, \gen{b}{L} \in \partialconjugations$ are such that $a \not \in L$ and $b \in K$ and $K \cap L = \emptyset$, then $\gen{a}{L} \in \partialconjugations$ and $[\genprod{a}{K}{L}, \gen{b}{L}] = 1$.
\end{lem}

\begin{proof}
Assume $\gen{a}{K}, \gen{b}{L} \in \partialconjugations$ are such that $a \not \in L$ and $b \in K$ and $K \cap L = \emptyset$.  Let $K'$ denote the connected component of $\Gamma \setminus \Star{a}$ such that $K' \cap L \neq \emptyset$.
Since $d(a, b) = 2$ and $a \not \in L$ and $b \not \in K'$ and $K' \cap L \neq \emptyset$, the pair $(\gen{a}{K'}, \gen{b}{L})$ falls into case (6) of Lemma \ref{Lemma:PCCases}.  Thus $K' = L$. The relation $[\genprod{a}{K}{L}, \gen{b}{L}] = 1$ is (2) in Theorem~\ref{Thm:Presentation}.
\end{proof}

\begin{cor}\label{Cor:CommutingRuleForInners}
If $a \in V \setminus Z$ and $\gen{b}{L} \in \partialconjugations$, then $[\inner{a}, \gen{b}{L}]=1$ if and only if  $a \not \in L$.
\end{cor}

 \section{The BNS-invariant of $P\Sigma(A)$}\label{main-a}

Throughout this section we consider an arbitrary non-trivial character $\chi\!:\! P\Sigma(A) \to \mathbb{R}$.  We write $\Sigma$ for $\Sigma^1(P\Sigma(A))$, and $\Sigma^c$ for the complement of $\Sigma$ in $S(P\Sigma(A))$.

\begin{lem}\label{Lemma:AtMostTwo}
Let $\gen{a}{K}, \gen{a}{L} \in \partialconjugations$ with $K \neq L$.
If $\gen{a}{K}, \gen{a}{L}$ and $\genprod{a}{K}{L}$ are $\chi$-hyperbolic, then $[\chi] \in \Sigma$.
\end{lem}

\begin{proof}
Suppose $\gen{a}{K}, \gen{a}{L}$ and $\genprod{a}{K}{L}$ are $\chi$-hyperbolic.  Consider a $P\Sigma (A)$-action on an $\mathbb{R}$-tree $T$ that realizes $\chi$.
Let $X = C_{\gen{a}{K}} = C_{\gen{a}{L}} = C_{\genprod{a}{K}{L}}$.  Let $\gen{c}{M}$ be an arbitrary partial conjugation.  If
$[\gen{a}{K}, \gen{c}{M}] =1$ or $[\gen{a}{K}, \gen{c}{M}] = 1$, then $X \subseteq C_{\gen{c}{M}}$ by Fact A; thus we may assume
$[\gen{a}{K}, \gen{c}{M}] \neq 1$
and $[\gen{a}{K}, \gen{c}{M}] \neq 1$.  It follows that $d(a, c) = 2$.  Since $K \cap L = \emptyset$, we may assume without loss of generality that $c \not \in K$.  Since $d(a, c) = 2$ and $c \not \in K$ and $[\gen{a}{K}, \gen{c}{M}] \neq 1$, the pair $(\gen{a}{K}, \gen{c}{M})$ falls into
case (3) or (6) of Lemma \ref{Lemma:PCCases}.

First consider the case that $(\gen{a}{K}, \gen{c}{M})$ falls into
case (3).  Then $a \in M$.  By Lemma \ref{Lemma:NewPC}, $\gen{c}{K} \in \partialconjugations$ and $[\genprod{c}{K}{M}, \gen{a}{K}]=1$.  By Fact A, $X \subseteq C_{\genprod{c}{K}{M}}$.  If $c \in L$, then $[\genprod{a}{K}{L}, \gen{c}{K}]=1$ and $X \subseteq C_{\gen{c}{K}}$ by Fact A.  By Fact B, $X \subseteq C_{\gen{c}{M}}$.  If $c \not\in L$, then the pair $(\gen{a}{L},\gen{c}{K})$ falls into case (5) of Lemma~\ref{Lemma:PCCases} which implies $[\gen{a}{L},\gen{c}{K}]=1$.  By Fact A, $X \subseteq C_{\gen{c}{K}}$ which implies $X \subseteq C_{\gen{c}{M}}$ by Fact B.

Now consider the case that $(\gen{a}{K}, \gen{c}{M})$ falls into case (6).  Then $a \not \in M$, $c \not \in K$ and $M = K$. Since $M \cap L = K \cap L = \emptyset$ and $a \not \in M$ and $[\gen{a}{L}, \gen{c}{M}] \neq 1$, the pair $(\gen{a}{L}, \gen{c}{M})$ falls into case (3) of Lemma \ref{Lemma:PCCases}.  Thus $c \in L$.
Since $c \in L$ and $M =K$, $[\genprod{a}{K}{L}, \gen{c}{M}] =1$, and by Fact A, $X \subseteq C_{\gen{c}{M}}$.
\end{proof}

\begin{cor}\label{Cor:ZeroOneOrTwo}
If $[\chi] \in \Sigma^c$, then the following properties hold for each vertex $a \in V \setminus Z$:
\begin{enumerate}
\item There are at most two $\chi$-hyperbolic partial conjugations with acting letter $a$.
\item \label{Property:HypInnerMeansOne} The inner automorphism $\inner{a}$ is $\chi$-hyperbolic if and only if there is exactly one $\chi$-hyperbolic partial conjugation with acting letter $a$.
\item \label{Property:CancelOut} If $\gen{a}{K}$ and $\gen{a}{L}$ are distinct $\chi$-hyperbolic partial conjugations, then $\chi(\gen{a}{K}) = - \chi(\gen{a}{L})$.
\end{enumerate}
\end{cor}

\begin{lem}\label{Lemma:Inner}
Let $\gen{a}{K}, \gen{a}{L} \in \partialconjugations$ with $K \neq L$, and let $b \in V$. If $\gen{a}{K}, \gen{a}{L}$ and $\inner{b}$ are $\chi$-hyperbolic, then $[\chi] \in \Sigma$.
\end{lem}

\begin{proof}
Suppose $\gen{a}{K}, \gen{a}{L}$ and $\inner{b}$ are $\chi$-hyperbolic.  If $a = b$, then $[\chi] \in \Sigma$ by Corollary~\ref{Cor:ZeroOneOrTwo}(\ref{Property:CancelOut}). Thus we may assume $b \neq a$.
Let $T$ be an $\mathbb{R}$-tree equipped with a $P\Sigma(A)$-action that realizes $\chi$.
Let $X = C_{\gen{a}{K}} = C_{\gen{a}{L}}$.
Since $\inner{b}$ is $\chi$-hyperbolic, there exists a connected component $M$ of $\Gamma \setminus \Star{b}$ such that $\gen{b}{M}$ is $\chi$-hyperbolic.
If $b \not \in K$, then $$[\gen{a}{K}, \inner{b}]=[\inner{b}, \gen{b}{M}] = 1;$$
if $b \in K$, then $b \not \in L$ and $$[\gen{a}{L}, \inner{b}]=[\inner{b}, \gen{b}{M}] = 1;$$  in either case, Fact A yields
$$C_{\gen{a}{L}} = C_{\inner{b}} = C_{\gen{b}{M}} = X.$$

Let $\gen{c}{N}$ be an arbitrary partial conjugation.  The lemma is proved if we show $X \subseteq C_{\gen{c}{ N}}$, for then the $P\Sigma(A)$-action fixes
$X$ setwise and is therefore not exceptional.  If $\gen{c}{N}$ commutes with any of the automorphisms $\gen{a}{K}, \gen{a}{L}, \gen{b}{M}$ or $\inner{b}$, then $X \subseteq C_{\gen{c}{N}}$ by Fact A.  Thus we may
assume $\gen{c}{N}$ commutes with none of these automorphisms.  It follows that $d(a, c) = d(b, c) = 2$ and $b \in N$.  Since $K \cap L = \emptyset$, we may assume
without loss of generality that $c \not \in L$.
We now consider cases based on whether or not $N$ contains $a$.

First we consider the case $a \in N$.  Since $b \in N$ and $c \not \in L$ and $[\gen{c}{N}, \gen{b}{L}] \neq 1$, the pair $(\gen{c}{N}, \gen{b}{L})$ falls into case (3) of Lemma \ref{Lemma:PCCases}; thus $N \cap L = \emptyset$.  By Lemma \ref{Lemma:NewPC}, $\gen{c}{L}$ is a partial conjugation, and $[\gen{a}{L}, \genprod{c}{L}{N}] = 1$.  By Fact A, $X \subseteq C_{\genprod{c}{L}{N}}$.  Since $b \in N$, $b \not \in L$, and by Corollary \ref{Cor:CommutingRuleForInners}, $[\inner{b}, \gen{c}{L}] = 1$.  By Fact A, $X \subseteq C_{\gen{c}{L}}$.
By Fact B, $X \subseteq C_{\gen{c}{N}}$.

Next we consider the case $a \not\in N$.  Since $a \not\in N$ and $c \not\in L$ and $[\gen{a}{L}, \gen{c}{N}] \neq 1$, the pair $(\gen{a}{L}, \gen{c}{N})$ falls into case (6) of
Lemma \ref{Lemma:PCCases}; thus $N = L$.  Let $N^{\prime}$ be the component of $\Gamma \setminus \Star{c}$ such that $a \in N^{\prime}$.  Therefore, $[\gen{a}{L}, \genprod{c}{N}{N^{\prime}}]=1$.  By Fact A, $X \subseteq C_{\genprod{c}{N}{N^{\prime}}}$.
Since $b \in N$, $b \not \in N^{\prime}$, and by Corollary \ref{Cor:CommutingRuleForInners}, $[\inner{b}, \gen{c}{N^{\prime}}] = 1$.
By Fact A, $X \subseteq C_{\gen{c}{N^{\prime}}}$. By Fact B, $X \subseteq C_{\gen{c}{N}}$.
\end{proof}

\begin{cor}\label{Cor:TypeIandTypeII}
If $[\chi] \in \Sigma^c$, then exactly one of the following holds:
\begin{enumerate}
\item [(I)] For each vertex $a \in V \setminus Z$, there is at most one $\chi$-hyperbolic
partial conjugation with acting letter $a$.
\item [(II)] For each vertex $a \in V \setminus Z$, $\inner{a}$ is $\chi$-elliptic and there are either zero or two $\chi$-hyperbolic
partial conjugations with acting letter $a$.
\end{enumerate}
\end{cor}

Motivated by the corollary above, we classify characters depending on which case, if any, they fall into.

\begin{defn}
We say $\chi$ is \emph{type I} if for each vertex $a \in V \setminus Z$, there is at most one $\chi$-hyperbolic
partial conjugation with acting letter $a$. We say $\chi$ is \emph{type II} if for each vertex $a \in V \setminus Z$, $\inner{a}$ is $\chi$-elliptic and there are either zero or two $\chi$-hyperbolic
partial conjugations with acting letter $a$.
\end{defn}

 \subsection{Characters of Type I}\label{typeI}

\begin{defn}[p-set]
A set of partial conjugations $\mathcal{Q} \subseteq \partialconjugations$ is a \emph{p-set} (or a \emph{partionable} set) if $\mathcal{Q}$ satisfies the following properties:
\begin{enumerate}
\item For each vertex $a \in V \setminus Z$, $\mathcal{Q}$ contains at most one partial conjugation with acting letter $a$.
\item The set $\mathcal{Q}$ admits a non-trivial partition $\{\mathcal{Q}_1, \mathcal{Q}_2\}$ with the property that $a \in L$ and $b \in K$ for each pair $(\gen{a}{K}, \gen{b}{L}) \in \mathcal{Q}_1 \times \mathcal{Q}_2$.
\end{enumerate}
We say $\{\mathcal{Q}_1, \mathcal{Q}_2\}$ is an \emph{admissible partition} of $\mathcal{Q}$.
\end{defn}

\begin{rem}
In the definition above, the first property is implied by the second.  In this instance we have preferred transparency to brevity.
\end{rem}

\begin{rem}\label{Remark:ConstructPSet}
An arbitrary maximal p-set $\mathcal{Q}$, and an admissible partition $\{\mathcal{Q}_1, \mathcal{Q}_2\}$ may be constructed as follows. Begin with a partial conjugation $\gen{a}{K}$.  Let $b_1, \dots, b_n$ be the vertices of $K$.  For $j = 1, \dots, n$, let
$L_j$ be the connected component of $\Gamma \setminus \Star{b_j}$ such that $a \in L_j$.  Let $a = a_1, a_2, \dots, a_m$ be the vertices of $\bigcap_{j=1}^n L_j \neq \emptyset.$
For $i = 1, 2, \dots, m$, let $K_i$ be the connected component of $\Gamma \setminus \Star{a_i}$ such that $b_1 \in K_i$.  Let
$$\mathcal{Q}_1 = \{\gen{a_1}{ K_1}, \dots, \gen{a_m}{ K_m}\}, \mathcal{Q}_2 = \{\gen{b_1}{ L_1}, \dots, \gen{b_n}{ L_n}\}, \text{ and } \mathcal{Q} = \mathcal{Q}_1 \cup \mathcal{Q}_2.$$
\end{rem}

\begin{prop}\label{Prop:ComplementIffPSet}
Suppose $\chi$ is type I and let $\mathcal{H}$ denote the set of $\chi$-hyperbolic partial conjugations.  Then $[\chi] \in \Sigma^c$ if and only if $\mathcal{H}$ is contained in some p-set $\mathcal{Q}$.
\end{prop}

\begin{proof}
Suppose $\chi$ is type I and $\mathcal{H}$ is contained in a p-set $\mathcal{Q}$.  Let $\{\mathcal{Q}_1, \mathcal{Q}_2\}$ be an admissible partition of $\mathcal{Q}$, with $$\mathcal{Q}_1 = \{\gen{a_1}{ K_1}, \dots, \gen{a_m}{ K_m}\} \text{ and } \mathcal{Q}_2 = \{\gen{b_1}{ L_1}, \dots, \gen{b_n}{ L_n}\}.$$  Let $G_1$ be the free abelian group with basis $\{u_1, \dots, u_m\}$, let $G_2$ be the free abelian group with basis $\{v_1, \dots, v_n\}$,
and let $G = G_1 \ast G_2$.  Consider a map such that: $\gen{a_i}{ K_i} \mapsto u_i$ for $i = 1 \dots, m$; $\gen{b_j}{ L_j} \mapsto v_j$ for $j = 1, \dots, n$; and all other partial conjugations are mapped to the identity.
It follows from Theorem \ref{Thm:Presentation} that this map determines an epimorphism $\phi\!:\! P\Sigma(A) \to G$.  Since $\chi$ factors through $\phi$, by Corollary~\ref{Corollary:FreeProduct}, $[\chi] \in \Sigma^c$.

Now suppose $\chi$ is type I and there is no p-set containing $\mathcal{H}$.  Let $T$ be an $\mathbb{R}$-tree equipped with a $P\Sigma(A)$-action that realizes $\chi$.
Let $\gen{a}{K} \in \mathcal{H}$, and let $X=C_{\gen{a}{K}}$. To prove the lemma it suffices to prove that $X \subseteq C_{\gen{c}{M}}$ for an arbitrary partial conjugation
$\gen{c}{M}$, because then we have that the action fixes $X$ setwise and hence is not exceptional.  If $\gen{c}{M}$ commutes with $\gen{a}{K}$, Fact A gives that $X \subseteq C_{\gen{c}{M}}$.  Thus we may assume that $\gen{c}{M}$ does not commute with $\gen{a}{K}$.

Next we show that the elements of $\mathcal{H}$ share the axis $X$.  Let $$\mathcal{I} = \{\gen{b}{L} \in \mathcal{H} \; | \; X \subseteq C_{\gen{b}{L}}\}.$$
Suppose $\mathcal{H} \neq \mathcal{I}$, and let $\gen{b}{L} \in \mathcal{H} \setminus I$. Since $X \not \nsubseteq C_{\gen{b}{L}}$, we have that $[\gen{b}{L}, \inner{a}] \neq 1$, and $[\gen{a}{K}, \inner{b}] \neq 1$.  By Corollary \ref{Cor:CommutingRuleForInners} we have $a \in L$ and $b \in K$.
It follows that $(\mathcal{I}, \mathcal{H} \setminus \mathcal{I})$ is an admissible partition, and $\mathcal{H}$ is a p-set---a contradiction which proves $\mathcal{H} =\mathcal{I}$.

Now consider an arbitrary partial conjugation such that $\gen{c}{M}$ does not commute with
$\gen{b}{L}$ or $\inner{b}$ whenever $\gen{b}{L} \in \mathcal{H}$.  It follows that $d(b, c) = 2$ for all $\gen{b}{K} \in \mathcal{H}$.  Since $\gen{c}{M}$ does not commute
with $\inner{b}$,  $b \in M$ for each $\gen{b}{L} \in \mathcal{H}$.
Since $\mathcal{H} \cup \{\gen{c}{M}\}$ is not a p-set, $\{\{\gen{c}{M}\}, \mathcal{H}\}$ is not an admissible partition.  Thus there exists $\gen{b}{L} \in \mathcal{H}$ such that $c \not \in L$.
Since $d(b, c) = 2$ and $b \in M$ and $c \not \in L$ and $[\gen{c}{M}, \gen{b}{L}] \neq 1$, Lemma \ref{Lemma:NewPC} gives that $\gen{c}{L}$ is a partial conjugation.  Since $[\genprod{c}{L}{M}, \gen{b}{L}]=1$, Fact A gives $X \subseteq C_{\genprod{c}{L}{M}}$.  By Corollary \ref{Cor:CommutingRuleForInners},  $[\gen{c}{L}, \inner{b}]=1$. By Fact A, $X \subseteq C_{\gen{c}{L}}$. By Fact B, $X \subseteq C_{\gen{c}{M}}$.
\end{proof}

 \subsection{Characters of Type II}\label{typeII}

\begin{defn}[$\delta$-p-set]
A set of partial conjugations $\mathcal{Q} \subseteq \partialconjugations$ is a \emph{$\delta$-p-set} if $\mathcal{Q}$ satisfies the following properties:
\begin{enumerate}
\item For each vertex $a \in V \setminus Z$, $\mathcal{Q}$ contains either zero or two partial conjugations with acting letter $a$.
\item The set $\mathcal{Q}$ admits a non-trivial partition $\{\mathcal{Q}_1, \mathcal{Q}_2\}$ such that $a \in L$ or $b \in K$ or $K = L$ for each pair $(\gen{a}{K}, \gen{b}{L}) \in \mathcal{Q}_1 \times \mathcal{Q}_2$.
\end{enumerate}
We say $\{\mathcal{Q}_1, \mathcal{Q}_2\}$ is an \emph{admissible $\delta$-partition} of $\mathcal{Q}$.
\end{defn}

\begin{rem}\label{Remark:EquivalentDelta}
It follows from the definitions that if $\gen{a}{K_1}, \gen{a}{K_{-1}} \in \mathcal{Q}$ and $K_1 \neq K_{-1}$, then either $\gen{a}{K_1}, \gen{a}{K_{-1}} \in \mathcal{Q}_1$ or
$\gen{a}{K_1}, \gen{a}{K_{-1}} \in \mathcal{Q}_2$.  Further, for each quadruple
$$(\gen{a}{K_1}, \gen{a}{K_{-1}}, \gen{b}{L_1}, \gen{b}{L_{-1}}) \in \mathcal{Q}_1 \times \mathcal{Q}_1 \times \mathcal{Q}_2 \times \mathcal{Q}_2,$$ $a \in L_i$ and $b \in K_j$ and $K_{-i} = L_{-j}$
for some $i, j \in \{-1, 1\}$.
\end{rem}

\begin{lem}\label{Lemma:PairsInC}
Let $\gen{a}{K_1}, \gen{a}{K_2}, \gen{b}{L_1}, \gen{b}{L_2} \in \partialconjugations$ be distinct partial conjugations.  Then
$[\gen{a}{K_i}, \gen{b}{L_j}] \neq 1$ for all $i, j \in \{1, 2\}$ if and only if $\{ \gen{a}{K_1}, \gen{a}{K_2},$ $\gen{b}{L_1}, \gen{b}{L_2} \}$ is a $\delta$-p-set.
\end{lem}

\begin{proof}
Assume $[\gen{a}{K_i}, \gen{b}{L_j}] \neq 1$ for all $i, j \in \{1, 2\}$.  Without loss of generality, assume $a \not \in L_2$ and $b \not \in K_2$.
Since $a \not \in L_2$ and $b \not \in K_2$ and $[\gen{a}{K_2}, \gen{b}{L_2}] \neq 1$, the pair $(\gen{a}{K_2}, \gen{b}{L_2})$ falls into case (6) of Lemma \ref{Lemma:PCCases}; thus $K_2 = L_2$.
Since $b \not \in K_2$ and $K_2 \cap L_1 = L_2 \cap L_1 = \emptyset$ and $[\gen{a}{K_2}, \gen{b}{L_1}] \neq 1$, the pair $(\gen{a}{K_2}, \gen{b}{L_1})$ falls into case (3) of Lemma \ref{Lemma:PCCases}; thus $a \in L_1$.
Since $a \not \in L_2$ and $K_1 \cap L_2 = K_1 \cap K_2 =\emptyset$ and $[\gen{a}{K_1}, \gen{b}{L_2}] \neq 1$, the pair $(\gen{a}{K_1}, \gen{b}{L_2})$ falls into case (3) of Lemma \ref{Lemma:PCCases}; thus $b \in K_1$.  Thus $\bigl\{\{ \gen{a}{K_1}, \gen{a}{K_2}\},  \{\gen{b}{L_1}, \gen{b}{L_2} \}\bigr\}$ is an admissible $\delta$-partition of $\{ \gen{a}{K_1}, \gen{a}{K_2}, \gen{b}{L_1}, \gen{b}{L_2} \}$.  The converse follows immediately from the definitions and Lemma \ref{Lemma:PCCases}.
\end{proof}

\begin{lem}\label{Lemma:HyperbolicPair}
Let $\gen{a}{K_1}, \gen{a}{K_2},\gen{c}{M}$ be distinct partial conjugations, and let $T$ be a $\mathbb{R}$-tree equipped with a $P\Sigma(A)$-action that realizes $\chi$.  If $\gen{a}{K_1}$ and $\gen{a}{K_2}$ are $\chi$-hyperbolic, $c \not \in K_1$ and $C_{\gen{a}{K_1}} \nsubseteq C_{\gen{c}{M}}$, then
$c \in K_2$ and $\gen{c}{K_1} \in \partialconjugations$.
\end{lem}

\begin{proof}
Suppose $\gen{a}{K_1}$ and $\gen{a}{K_2}$ are $\chi$-hyperbolic and $c \not \in K_1$.  Let $T$ be an $\mathbb{R}$-tree equipped with a $P\Sigma(A)$-action that realizes $\chi$, and suppose $C_{\gen{a}{K_1}} \nsubseteq C_{\gen{c}{M}}$. It follows that $d(a, c) = 2$.

Since $c \not \in K_1$ and $[\gen{a}{K_1}, \gen{c}{M}] \neq 1$, the pair $(\gen{a}{K_1}, \gen{c}{M})$ falls into either case (3) or case (6) of Lemma \ref{Lemma:PCCases}.
If $(\gen{a}{K_1}, \gen{c}{M})$ falls into case (3), $a \in M$.  By Lemma \ref{Lemma:NewPC}, $\gen{c}{K_1} \in \partialconjugations$.  Since $[\genprod{c}{K_1}{M}, \gen{a}{K_1}] = 1$, but Fact B cannot be used, we must have that
$[\gen{c}{K_1}, \gen{a}{K_2}] \neq 1$; thus $(\gen{c}{K_1}, \gen{a}{K_2})$ falls into case (3) of Lemma \ref{Lemma:PCCases}, and $c \in K_2$.
If $(\gen{a}{K_1}, \gen{c}{M})$ falls into case (6), we have $a \not \in M$ and $M = K_1$.  But then since $a \not \in M$ and $M \cap K_2 = \emptyset$ and $[\gen{a}{K_2}, \gen{c}{M}] \neq 1$,
the pair $(\gen{a}{K_2}, \gen{c}{M})$ falls into case (3) of Lemma \ref{Lemma:PCCases}.  Thus $c \in K_2$.
\end{proof}

\begin{prop}\label{Prop:ComplementIffDelta}
Suppose $\chi$ is type II and let $\mathcal{H}$ denote the set of $\chi$-hyperbolic partial conjugations.  Then $[\chi] \in \Sigma^c$ if and only if $\mathcal{H}$ is contained in some $\delta$-p-set $\mathcal{Q}$.
\end{prop}

\begin{proof}
Suppose $\mathcal{H}$ is contained in some $\delta$-p-set $\mathcal{Q}$.  Let $\{ \mathcal{Q}_1, \mathcal{Q}_2 \}$ be an admissible partition of $\mathcal{Q}$ with $$\mathcal{Q}_1 = \{ \gen{a_1}{K_1},\gen{a_1}{L_1}, \ldots, \gen{a_m}{K_m},\gen{a_m}{L_m} \} \text{ and } \mathcal{Q}_2 = \{ \gen{b_1}{M_1},\gen{b_1}{N_1}, \ldots, \gen{b_n}{M_n}, \gen{b_n}{N_n} \}.$$  Let $G_1$ be the free abelian group with basis $\{ u_1, \ldots, u_m \}$, $G_2$ be the free abelian group with basis $\{ v_1, \ldots, v_n \}$, and $G = G_1 \ast G_2$.  Define $\phi:P\Sigma(A) \to G$ by $\gen{a_i}{K_i} \mapsto u_i$ and $\gen{a_i}{L_i} \mapsto u_i^{-1}$ for $i = 1, \ldots, m$, $\gen{b_j}{M_j} \mapsto v_j$ and $\gen{b_j}{N_j} \mapsto v_j^{-1}$ for $j = 1, \ldots, n$, and all other generators map to the identity.  For $\gen{a_i}{K_i} \in \mathcal{Q}_1$ and $\gen{b_j}{M_j} \in \mathcal{Q}_2$, we have either $a_i \in M_j$ or $K_i = M_j$, and in either case, $[\gen{a_i}{ K_i},\gen{b_j}{M_j}] \neq 1$.  Thus, $\phi$ is a well-defined epimorphism.  Since $\chi$ factors through this map, by Corollary~\ref{Corollary:FreeProduct}, we have $[\chi] \in \Sigma^c$.

Suppose $\mathcal{H}$ is not contained in some $\delta$-p-set $\mathcal{Q}$.  Let $T$ be an $\mathbb{R}$-tree equipped with an $P\Sigma(A)$-action that realizes $\chi$.  Since $\chi$ is type II, we have $\pi_{a,K}, \pi_{a,L} \in \mathcal{H}$ for some vertex $a \in V \setminus Z$. Let $X = C_{\gen{a}{K}} = C_{\gen{a}{L}}$.

First we will show $X = C_{\gen{b}{M}}$ for each $\gen{b}{M} \in \mathcal{H}$.  Define $\mathcal{I} = \{ \gen{b}{M} \in \mathcal{H}~|~X = C_{\gen{b}{M}} \}$.  Assume $\mathcal{H} \neq \mathcal{I}$, and let $\gen{b}{M} \in \mathcal{H} \setminus \mathcal{I}$.  Since $\gen{b}{M} \in \mathcal{H}$, there exists $\gen{b}{N} \in \mathcal{H}$ where $M \neq N$, and clearly $\gen{b}{N} \in \mathcal{H} \setminus \mathcal{I}$.  Let $\gen{c}{Q} \in \mathcal{I}$.  Again, there must be $\gen{c}{R} \in \mathcal{I}$ such that $Q \neq R$.  By Lemma~\ref{Lemma:PairsInC}, $(\mathcal{I},\mathcal{H} \setminus \mathcal{I})$ is an admissible $\delta$-partition which is a contradiction, so $\mathcal{H} = \mathcal{I}$.

Now let $\gen{b}{M}$ be an arbitrary element of $\partialconjugations$, and let
$$\mathcal{H} = \{ \gen{a_1}{K_1}, \gen{a_1}{L_1}, \ldots, \gen{a_m}{K_m}, \gen{a_m}{L_m} \}.$$
By Lemma~\ref{Lemma:HyperbolicPair},
either $X \subseteq C_{\gen{b}{M}}$ or without loss of generality, $b \in K_i$ and $\gen{b}{L_i} \in \partialconjugations$ for each $i = 1, \ldots, m$.  Assume the latter is true, so either $a_i \not\in M$ for some $i \in \{1, \ldots, m \}$ or $a_i \in M$ for each $i \in \{ 1, \ldots, m \}$.  If $a_i \not\in M$, then $\gen{b}{M}$ commutes with $\gen{a_i}{L_i}$ which implies by Fact A that $X \subseteq C_{\gen{b}{M}}$.  Suppose for each $i = 1, \ldots, m$, $a_i \in M$.  If $L_i \cap L_j = \emptyset$ for some $i \neq j$, then $[\gen{b}{L_i},\gen{a_j}{L_j}]=1$ which implies $X \subseteq C_{\gen{b}{L_i}}$.  Since $a \in M$ and $b \not\in L_i$ and $L_i \cap M = \emptyset$, we have $[\genprod{b}{L_i}{M}, \gen{a_i}{L_i}]=1$.  By Fact A, $X \subseteq C_{\genprod{b}{L_i}{M}}$, and by Fact B, $X \subseteq C_{\gen{b}{M}}$.  Suppose $L_i \cap L_j \neq \emptyset$ for each pair $(i,j)$.  Then $L_i = L_j$ for each pair $(i,j)$ since these are connected components of $\Gamma \setminus st(b)$.  Denote by $L$ this connected component.  Then $(\{ \gen{b}{M}, \gen{b}{L} \}, \mathcal{H})$ is an admissible partition of the $\delta$-p-set $\mathcal{H} \cup \{ \gen{b}{M}, \gen{b}{L} \}$
which is a contradiction.  Therefore, $X \subseteq C_{\gen{b}{M}}$, and $[\chi] \in \Sigma$.
\end{proof}

Proposition~\ref{Prop:ComplementIffPSet} and Proposition~\ref{Prop:ComplementIffDelta} prove our first main theorem.

\setcounter{mainthm}{0}
\begin{mainthm}\label{Theorem:MainA}
Let $\chi:P\Sigma(A) \to \mathbb{R}$ be a character, and let $\mathcal{H}$ denote the set of $\chi$-hyperbolic partial conjugations.  Then $[\chi] \in \Sigma^c$ if and only if $\mathcal{H}$ is contained in a set of partial conjugations $\mathcal{Q}$ such that either:
\begin{enumerate}
\item The set $\mathcal{Q}$ admits a partition $\{\mathcal{Q}_1, \mathcal{Q}_2\}$ with the property that $a \in L$ and $b \in K$ for each pair $(\gen{a}{K}, \gen{b}{L}) \in \mathcal{Q}_1 \times \mathcal{Q}_2$; or
\item For each vertex $a \in V \setminus Z$, $\inner{a}$ is $\chi$-elliptic, and $\mathcal{Q}$ contains either zero or two partial conjugations with acting letter $a$; and $\mathcal{Q}$ admits a partition $\{\mathcal{Q}_1, \mathcal{Q}_2\}$ with the property that $a \in L$ or $b \in K$ or $K=L$ for each pair $(\gen{a}{K}, \gen{b}{L}) \in \mathcal{Q}_1 \times \mathcal{Q}_2$.
\end{enumerate}
\end{mainthm}

\begin{exam}\label{Example:RAAG}
Let $A = \langle a, b, c, d, e~|~[a,b], [b,c], [c,d], [c,e] \rangle$.  The pure symmetric automorphism group $P\Sigma(A)$ is generated by the set $$\{ \gen{a}{\{c,d,e\}}, \gen{b}{\{d\}}, \gen{b}{\{e\}}, \gen{c}{\{a\}}, \gen{d}{\{a,b\}}, \gen{d}{\{e\}}, \gen{e}{\{a,b\}}, \gen{e}{\{d\}}\},$$ so $S(P\Sigma(A))$ is a $7$-dimensional sphere.  The maximal p-sets are:
\begin{enumerate}
\item $\mathcal{Q}_1 = \{ \gen{a}{\{c,d,e\}}, \gen{c}{\{a\}}, \gen{d}{\{a,b\}}, \gen{e}{\{a,b\}} \}$ with admissible partition $\{  \gen{a}{\{c,d,e\}} \}$ and $\{ \gen{c}{\{a\}}, \gen{d}{\{a,b\}}, \gen{e}{\{a,b\}} \}$,
\item $\mathcal{Q}_2 = \{ \gen{a}{\{c,d,e\}}, \gen{b}{\{d\}}, \gen{d}{\{a,b\}} \}$ with admissible partition $\{ \gen{a}{\{c,d,e\}}, \gen{b}{\{d\}} \}$ and $\{ \gen{d}{\{a,b\}} \}$,
\item $\mathcal{Q}_3 = \{ \gen{a}{\{c,d,e\}}, \gen{b}{\{e\}}, \gen{e}{\{a,b\}} \}$ with admissible partition $\{ \gen{a}{\{c,d,e\}}, \gen{b}{\{e\}} \}$ and $\{ \gen{e}{\{a,b\}} \}$, and
\item $\mathcal{Q}_4 = \{ \gen{d}{\{e\}}, \gen{e}{\{d\}}\}$
\end{enumerate}

The only maximal $\delta$-p-set is $\{ \gen{b}{\{d\}}, \gen{b}{\{e\}}, \gen{d}{\{a,b\}}, \gen{d}{\{e\}}, \gen{e}{\{a,b\}}, \gen{e}{\{d\}}\}$ with admissible partition $\{ \gen{b}{\{d\}}, \gen{b}{\{e\}} \}$ and $\{ \gen{d}{\{a,b\}}, \gen{d}{\{e\}}, \gen{e}{\{a,b\}}, \gen{e}{\{d\}}\}$.  Therefore, $\Sigma^c$ consists of the characters $[\chi]$ such that:
\begin{enumerate}
\item $\chi$ sends all generators to zero except maybe those generators in $\mathcal{Q}_i$ for some $1 \leq i \leq 4$, or
\item $\chi(\gen{b}{\{d\}}) = -(\gen{b}{\{e\}}), \chi(\gen{d}{\{a,b\}}) = -\chi(\gen{d}{\{e\}}), \chi(\gen{e}{\{a,b\}}) = -\chi(\gen{e}{\{d\}})$, and $\chi$ sends all other generators to zero.
\end{enumerate}
\end{exam}

 \section{Right-angled Artin groups with separating intersecting links}\label{sils}

A graph $\Gamma$ has a separating intersection of links (SIL) if there exists a pair $a, b$ of distinct non-adjacent vertices such that $\Gamma\setminus (\Link{a} \cap \Link{b})$ has a connected component $M$ containing neither $a$ nor $b$.
The following proposition was proven in \cite{charney}, and we state the result in terms of our particular circumstance.

\begin{prop}\label{Prop:NoSILs}\cite[Theorem 3.6]{charney}
If the defining graph $\Gamma$ contains no SILs, then $P\Sigma(A)$ is isomorphic to a right-angled Artin group.
\end{prop}

In this section we prove the converse to Proposition \ref{Prop:NoSILs}, which completes the proof of Theorem \ref{Theorem:MainB}.  We continue to use the notation described above.

Given a non-trivial character $\psi:A \to \mathbb{R}$, we write $\Gamma_{\psi}$ for the full subgraph of $\Gamma$ spanned by the set of $\psi$-hyperbolic vertices.  The subgraph $\Gamma_\psi$ is called \emph{dominating} if every vertex in $\Gamma$ is either in, or adjacent to a vertex in, $\Gamma_{\psi}$.  It was shown in \cite{meier} that:

\begin{thm}\label{Theorem:SigmaRAAGs}\cite[Theorem 4.1]{meier}
Suppose $[\psi] \in S(A)$.  Then $[\psi] \in \Sigma^1(A)$ if and only if $\Gamma_{\psi}$ is connected and dominating.
\end{thm}

For each set of vertices $U \subseteq V$, we write $S(U)$ for the sub-sphere
$$\{[\psi] \in S(A) \; | \; \psi(v) = 0 \text{ for all } v \in V \setminus U\}.$$
We note that $S(U)$ is a sub-sphere of dimension $|U|-1$ (we consider $S(\emptyset)$ to be a sub-sphere of dimension $-\!1$).  We say $S(U)$ is a \emph{missing sub-sphere} if $S(U) \subseteq \Sigma(A)^c$, and we note this holds exactly when the full subgraph spanned by $U$ is disconnected or non-dominating.  If $U$ spans a subgraph of $\Gamma$ which is non-dominating, then every subset of $U$ spans a subset of $\Gamma$ which is non-dominating; if $U$ spans a subgraph of $\Gamma$ which is disconnected, then every subset of $U$ spans a subset of $\Gamma$ which is disconnected or non-dominating.
It follows that if $S(U)$ and $S(W)$ are missing sub-spheres, then $S(U \cap W)$ is a missing sub-sphere.  It also follows that $\Sigma^1(A)$ is constructed from $S(A)$ by removing the maximal missing sub-spheres.  Viewing the construction of $\Sigma^1(A)$ in this distinctive way, we observe the following:

\begin{lem}\label{Lem:CountingDimensions}
If $A$ is a right-angled Artin group, and $S_1, \dots, S_p \subseteq S(A)$ are the maximal missing sub-spheres, then
\begin{multline*}
\rank(A/[A,A]) - \rank(Z(A))=  1 + \sum_{i} \dim(S_i) - \sum_{i < j} \dim(S_i \cap S_j)  \\ + \sum_{i<j<k} \dim(S_i \cap S_j \cap S_k) -  \dots+ (-1)^{n-1} \dim(S_1 \cap \dots \cap S_p).
\end{multline*}

\end{lem}

\begin{proof}
Since $\rank(A/[A,A]) = |V|$, and $\rank(Z(A)) = |Z|$, the lemma is proved if we show that the right-hand side of the equation sums to $|V \setminus Z|$.  It follows from Theorem \ref{Theorem:SigmaRAAGs} that, for each $i$, $S_i = S(U_i)$ for some maximal set of vertices $U_i$ which spans a disconnected or non-dominating subgraph of $\Gamma$.  For each vertex $v \in V \setminus Z$, the singleton set $\{v\}$ spans a non-dominating subgraph of $\Gamma$, and hence $v$ is contained in at least one set $U_i$. Any set of vertices containing an element of $Z$ spans a connected and dominating subgraph of $\Gamma$.  Thus we have $V \setminus Z = U_1 \cup U_2 \cup \dots \cup U_p$.  Now the Principle of Inclusion-Exclusion, together with the identity $\sum_{i=1}^p (-\!1)^{i-1}{p \choose i} = 1$, gives:
\begin{eqnarray*}
&  & |U_1 \cup U_2 \cup \dots \cup U_p|\\ &&\\
& = & \sum_{i} |U_i| - \sum_{i < j} |U_i \cap U_j| + \sum_{i<j<k} |U_i \cap U_j \cap U_k| - \dots \\ 
&& \ldots + (-1)^{n-1} |U_1 \cap \dots \cap U_p|\\&&\\
& = & \sum_{i} \bigl(\dim(S_i)+1\bigr) - \sum_{i < j} \bigl(\dim(S_i \cap S_j)+1\bigr) \\
&& + \sum_{i<j<k} \bigl(\dim(S_i \cap S_j \cap S_k)+1\bigr) - \ldots \\
&& \dots+ (-1)^{n-1} \bigl(\dim(S_1 \cap \dots \cap S_p)+1\bigr)\\&&\\
&=& 1+ \sum_{i} \dim(S_i) - \sum_{i < j} \dim(S_i \cap S_j) \\
&& + \sum_{i<j<k} \dim(S_i \cap S_j \cap S_k) - \dots+ (-1)^{n-1} \dim(S_1 \cap \dots \cap S_p).
\end{eqnarray*}
\end{proof}

Next we characterize the maximal missing sub-spheres in $S(A)$ by a property which makes no reference to the canonical generating set of $A$, thereby allowing us to identify the only candidates for maximal missing sub-spheres in $S(G)$ when we do not yet know whether or not $G$
is a right-angled Artin group.

A normal subgroup $K$ in a finitely-generated group $G$ is a \emph{complement kernel} if $K = \ker(\psi)$ for some $[\psi] \in \Sigma(G)^c$.  For such $K$, the set $$\{[\psi] \in \Sigma^1(G)^c \; | \; K \subseteq \ker(\psi)\}$$ is the \emph{complement subspace determined by $K$}.

\begin{lem}
For each subset $S \subseteq S(A)$, $S$ is a maximal missing sub-sphere if and only if $S$ is the complement subspace determined by some minimal complement kernel $K$.
\end{lem}

\begin{proof}
Suppose $S = S(U)$ is a maximal missing sub-sphere in $S(A)$, with $U = \{u_1, \dots, u_p\}$.  Let $\psi_U\!:\! A \to \mathbb{R}$ denote the character such that $$\psi_U(v) = 0 \text{ for } v \in V \setminus U, \text{ and } \psi_U(u_i) = \pi^i \text{ for } i = 1, \dots, p.$$
Since $\pi$ is transcendental, $K_U = \ker(\psi_U)$ consists of those elements $a \in A$ with zero exponent sums in each of the vertices $u_1, \dots, u_p$. It follows that $[\psi_U] \in S(U)$, and $K_U \subseteq \ker(\psi)$ for every $[\psi] \in S(U)$. Thus $S(U)$ is the complement subspace determined by $K_U$.
The maximality of $U$, together with Theorem \ref{Theorem:SigmaRAAGs}, implies that $K_U$ is minimal amongst the kernels of characters in $\Sigma^1(A)^c$.  It also follows from Theorem \ref{Theorem:SigmaRAAGs} that every minimal complement kernel arises in this way.
\end{proof}

We now have an approach for showing that a finitely-generated torsion-free group $G$ is not a right-angled Artin group: we identify the minimal complement kernels $K_1, \dots, K_p$ in $G$; use these to identify the corresponding complement subspaces $S_1, \dots, S_p$ in $S(G)$; then show that Lemma \ref{Lem:CountingDimensions} fails.  We carry out this plan for $P\Sigma(A)$ when $\Gamma$ contains a SIL.

\begin{lem}
If $S$ is the complement subspace corresponding to a minimal complement kernel $K$ in $P\Sigma(A)$, then either:
$$\displaystyle S = \{[\chi] \in S(P\Sigma(A)) \; | \; \chi(\gen{a}{K}) = 0 \text{ for all } \gen{a}{K} \in \partialconjugations \setminus \mathcal{Q}\}$$ for some maximal p-set $\mathcal{Q}$, in which case $\dim(S) = |\mathcal{Q}|-1$; or
$$\displaystyle S = \{[\chi] \in S(A) \; | \; \chi(\gen{a}{K}) = 0 \text{ for all } \gen{a}{K} \in \partialconjugations \setminus \mathcal{Q}, \text{ and } \chi(\inner{v}) = 0 \text{ for all } v \in V\}$$ for some maximal $\delta$-p-set $\mathcal{Q}$, in which case $\dim(S) = |\mathcal{Q}|/2-1$.
\end{lem}

\begin{proof}
Suppose $S$ is the complement subspace corresponding to a minimal complement kernel $K$ in $P\Sigma(A)$, and let $\chi\!:\! P\Sigma(A) \to \mathbb{R}$ be a character with kernel $K$.  By Corollary \ref{Cor:TypeIandTypeII}, $\chi$ is type I or type II.

Consider first the case that $\chi$ is type I.  By Proposition \ref{Prop:ComplementIffPSet}, the $\chi$-hyperbolic vertices comprise a p-set $\mathcal{Q}$.  The minimality of $K$ implies that $\mathcal{Q}$ is not contained in a larger p-set. That $S$ is as described follows immediately.

Now consider the case that $\chi$ is type II.  By Proposition \ref{Prop:ComplementIffDelta}, the $\chi$-hyperbolic vertices comprise a $\delta$-p-set $\mathcal{Q}$.  The minimality of $K$ implies that $\mathcal{Q}$ is not contained in a larger $\delta$-p-set. That $S$ is as described follows immediately.

\end{proof}

\begin{lem}\label{Lem:CountingInPSigmaA}
If $\mathcal{Q}_1, \dots, \mathcal{Q}_p$ are the maximal p-sets in $P\Sigma(A)$, and $S_1, \dots, S_p$ the corresponding complement subspaces, then
\begin{multline*}
\rank(P\Sigma(A)/[P\Sigma(A), P\Sigma(A)])=  1 + \sum_{i} \dim(S_i) - \sum_{i < j} \dim(S_i \cap S_j)  \\ + \sum_{i<j<k} \dim(S_i \cap S_j \cap S_k) -  \dots+ (-1)^{n-1} \dim(S_1 \cap \dots \cap S_p).
\end{multline*}
\end{lem}

\begin{proof}
It follows from Theorem \ref{Thm:Presentation} that $\rank\bigl(P\Sigma(A)/[P\Sigma(A), P\Sigma(A)]\bigr) = |\partialconjugations|$.  Suppose $\gen{a}{K} \in \partialconjugations$.  Let $b$ be a vertex in $K$, and let $L$ be the connected component of $\Gamma \setminus \Star{b}$ such that $a \in L$.  Then  $\{\gen{a}{K}, \gen{b}{L}\}$ is a p-set.  Thus every partial conjugation is contained in at least one p-set.  Now, as in the proof of Lemma \ref{Lem:CountingDimensions}, the lemma follows from the Principle of Inclusion-Exclusion and the identity $\sum_{i=1}^p (-\!1)^{i-1}{p \choose i} = 1$.
\end{proof}

\begin{cor}\label{Cor:NoTypeII}
If $P\Sigma(A)$ is isomorphic to a right-angled Artin group, then $\Sigma^1(P\Sigma(A))^c$ contains no characters of type II.
\end{cor}

\begin{proof}
Suppose $P\Sigma(A)$ is isomorphic to a right-angled Artin group, and assume the notation of Lemma \ref{Lem:CountingInPSigmaA}.  It follows from Theorem \ref{Lem:CountingDimensions} and
Lemma \ref{Lem:CountingInPSigmaA}, that $S_1, \dots, S_p$ is the complete list of complement subspaces corresponding to minimal complement kernels (and $P\Sigma(A)$ has no center).  Thus $S_1, \dots, S_p$ is the complete list of maximal missing sub-spheres in $S(P\Sigma(A))$, and
$$\Sigma^1(P\Sigma(A))^c = \bigcup_{i=1}^p S_i.$$
Since each character in each $S_i$ is type I, and by Corollary \ref{Cor:TypeIandTypeII} no character is type I and type II, we conclude that $\Sigma^1(P\Sigma(A))^c$ contains no characters of type II.
\end{proof}


\begin{prop}\label{Prop:Converse}
If $\Gamma$ contains a SIL, then $P\Sigma(A)$ is not isomorphic to a right-angled Artin group.
\end{prop}

\begin{proof}

Suppose $\Gamma$ contains a SIL.  Let $a, b$ and $M$ be as in the definition of a SIL, let $K$ be the connected component of $\Gamma \setminus \Star{a}$ that contains $b$, and let $L$ be the connected component of $\Gamma \setminus \Star{b}$ that contains $a$.  The set $\{\gen{a}{K}, \gen{a}{M}, \gen{b}{L}, \gen{b}{M}\}$ is a $\delta$-p-set.  In particular, $\Sigma^1(P\Sigma(A))$ contains at least one character of type II.  By Corollary \ref{Cor:NoTypeII}, $P\Sigma(A)$ is not isomorphic to a right-angled Artin group.
\end{proof}

Proposition~\ref{Prop:Converse} and \cite[Theorem 3.6]{charney} prove Theorem~\ref{Theorem:MainB}.

\bibliography{sigma-psa-raags}
\bibliographystyle{amsalpha}
\end{document}